\documentclass[aip,jcp,reprint]{revtex4-2}
\usepackage{mathtools}
\usepackage{amsmath}
\usepackage{amssymb}
\usepackage{amsthm}
\usepackage{pgfplots}\pgfplotsset{compat=1.18} 
\usepackage{tikz}\usetikzlibrary{shapes.geometric,arrows.meta,tikzmark}
\usepackage{svg}
\usepackage{mathdots}
\usepackage{siunitx}
\usepackage{algpseudocode}
\usepackage{algorithm}
\usepackage{enumitem}

\tikzstyle{startstop} = [circle, draw=black, fill=lightred, align=center]
\tikzstyle{basicblock} = [rectangle, rounded corners, draw=black, fill=lightblue]
\tikzstyle{decision} = [diamond, draw=black, fill=lightviolet, align=center]
\tikzstyle{thickarrow} = [thick, -{Stealth[scale=1]}]
\tikzstyle{arrow} = [-{Stealth[scale=1.2]}]

\newtheorem{theorem}{Theorem}

\definecolor{lightred}{rgb}{1,0.8,0.8}
\definecolor{lightblue}{rgb}{0.8,1,1}
\definecolor{lightviolet}{rgb}{1,0.8,1}

\algrenewcommand\algorithmicdo{}
\algrenewcommand\algorithmicthen{}

\setlist[description]{leftmargin=0pt,labelindent=0pt}

\makeatletter
\def\@email#1#2{%
 \endgroup
 \patchcmd{\titleblock@produce}
  {\frontmatter@RRAPformat}
  {\frontmatter@RRAPformat{\produce@RRAP{*#1\href{mailto:#2}{#2}}}\frontmatter@RRAPformat}
  {}{}
}%
\makeatother

\begin{document}

\title{Efficient Boys function evaluation using minimax approximation}
\author{Rasmus Vikhamar-Sandberg}
\email{rasmus.vikhamar-sandberg@uit.no, michal.repisky@uit.no}
\affiliation{Department of Physical and Theoretical Chemistry, Faculty of Natural Sciences, Comenius University, SK-84215 Bratislava, Slovakia}
\affiliation{Hylleraas Center for Quantum Molecular Sciences, Department of Chemistry, UiT The Arctic University of Norway, N-9037 Tromsø , Norway}
\author{Michal Repisky}
\affiliation{Department of Physical and Theoretical Chemistry, Faculty of Natural Sciences, Comenius University, SK-84215 Bratislava, Slovakia}
\affiliation{Hylleraas Center for Quantum Molecular Sciences, Department of Chemistry, UiT The Arctic University of Norway, N-9037 Tromsø , Norway}

\begin{abstract}
We present an algorithm for efficient evaluation of Boys functions $F_0,\dots,F_{k_\mathrm{max}}$ tailored to modern computing architectures, in particular graphical processing units (GPUs), where maximum throughput is high and data movement is costly. The method combines rational minimax approximations with upward and downward recurrence relations. The non-negative real axis is partitioned into three regions, $[0,\infty\rangle = A\cup B\cup C$, where regions $A$ and $B$ are treated using rational minimax approximations and region $C$ by an asymptotic approximation. This formulation avoids lookup tables and irregular memory access, making it well suited hardware with high maximum throughput and low latency. The rational minimax coefficients are generated using the rational Remez algorithm. For a target maximum absolute error of $\varepsilon_\mathrm{tol} = \num{5e-14}$, the corresponding approximation regions and coefficients for Boys functions 
$F_0,\dots,F_{32}$ are provided in Appendix~\ref{coefficients_and_regions}.
\end{abstract}

\maketitle

\section{Introduction}
One of the major computational steps in \textit{ab initio} quantum chemical calculations with Gaussian-type orbitals (GTOs) as basis is the evaluation of two-electron repulsion integrals (ERIs)~\cite{Boys1950,Roothaan1951,Hall1951}
\begin{equation}
   \int d\mathbf{r}_{1} \int d\mathbf{r}_{2} \ \frac{\phi_\mu^\dagger(\mathbf{r}_{1})\phi_\nu(\mathbf{r}_{1})\phi_\kappa^\dagger(\mathbf{r}_{2})\phi_\lambda(\mathbf{r}_{2})}{|\mathbf{r}_{1}-\mathbf{r}_{2}|},
\end{equation}
where cartesian GTOs are functions of the form
\begin{equation}
    \phi_\mu(\mathbf{r}) = C(x-A_x)^{l_x}(y-A_y)^{l_y}(z-A_z)^{l_z}e^{-\alpha(\mathbf{r}-\mathbf{A})^2}.
\end{equation}
centered at $\mathbf{A}=(A_x,A_y,A_z)$ and characterized by positive exponent factor $\alpha$. These integrals as well as other similar electronic integrals can be expressed exactly in terms of Boys functions~\cite{Boys1950}, a family of functions defined by
\begin{equation}
    F_k(x) := \int_0^1 dt \ t^{2k}e^{-xt^2}
\end{equation}
for all $x\in[0,\infty\rangle$ and $k\in\mathbb{N}_0$. For GTOs with low angular momentum $l = l_x+l_y+l_z\in\mathbb{N}_0$, a significant amount of time of the integral evaluation is spent on evaluating Boys functions~\cite{Hamilton1991}.

A widely used strategy for evaluating the Boys function is pre-tabulation, wherein function values are stored on a grid and interpolated as needed~\cite{Reine2011,helgaker2013molecular}. Following the seminal work of Shavitt~\cite{shavitt1963gaussian}, the pre-tabulation methodology was successfully used by McMurchie and Davidson~\cite{MCMURCHIE1978218}, and Obara and Saika~\cite{Obara1986}. Related approaches were later proposed by Gill et al.~\cite{Gill1991}, Ishida~\cite{ishida1996ace}, or Weiss and Ochsenfeld~\cite{Weiss2015}, with the goal to minimize the number of floating-point operations. The method uses little arithmetic, but has an irregular memory access pattern because of random look-ups of the table of Boys function values. Irregular memory access pattern has generally a bad impact on the performance of the code.

In response to this shortcoming, Mazur et al. published an approach where rational minimax approximations were used as a more memory-light method at the expense of being more arithmetic-heavy, thereby eliminating the irregular memory access pattern~\cite{Mazur2016}. Rational minimax approximations are by definition best rational approximations with a fixed numerator and denominator degree. Other methods that require little data are approximation by Chebyshev interpolation, polynomial minimax approximation and Padé approximation. The use of rational minimax approximation can be further motivated by the following argument: Under the assumption that the hardware can only perform the four arithmetic operations $+$,$-$,$\cdot$,$/$, (which is wrong) the only functions we can evaluate are rational functions, and consequently one can show that rational minimax approximation is the method that minimizes the number of arithmetic operations given an error threshold, see Section \ref{walshtable}. Furthermore, rational minimax approximations typically outperforms polynomial minimax approximations when the function has near singularities [\onlinecite{trefethen2019approximation}, chap. 23] despite the fact that division is a computationally expensive operation.

It is also worth mentioning Schaad and Morrell~\cite{Schaad1971} who were early out to generate rational minimax approximations of Boys functions, as well as Laikov~\cite{Laikov2025} who has recently conducted work that is similar but independent to this work with the addition of single instruction, multiple data (SIMD) vectorizibility. Mazur et al. provided rational minimax approximations only for Boys functions up to $F_8$. This is insufficient for routine quantum chemical calculations. They used the implementation of Remez' algorithm provided by the Boost C++ library~\cite{BoostC++} to generate the minimax approximations and reported that a damping factor of the form $e^{ax}$ was necessary, due to convergence issues in the Boost C++ implementation.

Several libraries for generation of rational minimax approximation were tried, AlgRemez~\cite{AlgRemez}, Baryrat~\cite{Baryrat} and Chebfun~\cite{Chebfun}. The Boost C++ library was omitted, as Mazur et al. required a damping factor to achieve convergence. AlgRemez did not converge for Boys functions. Baryrat, which is based on the BRASIL algorithm~\cite{BRASIL} instead of Remez' algorithm, converged only for short intervals as domain of the minimax approximations. In addition, it was impractically slow. In contrast, the Chebfun library provides a robust and efficient implementation of Remez' algorithm. Unfortunately, its routines outputs functions in the form of "chebfuns" which are piecewise Chebyshev interpolations. When extracting the numerator and denominator coefficients of the the chebfun in standard polynomial basis, an essential amount of precision is lost, making it impossible to reach an absolute error threshold of \num{5e-14}.

Therefore, a custom implementation of Remez' algorithm was made. The code is open source and accessed at \url{gitlab.com/rasmusvi/minimax}. It was implemented in quadruple precision since the outputted minimax coefficients should be in double precision for applications in quantum chemistry.

With this implementation, one can successfully generate near-correctly rounded double precision coefficients of rational minimax approximations of Boys functions of any order and error threshold as long as the minimax approximation is not too close to having a \emph{defect} and as long as the minimax coefficients with respect to standard polynomial basis $\left\{1,x,x^2,\dots\right\}$ are well-conditioned. Well-conditioned coefficients mean that a small perturbation in the coefficients gives a small perturbation the polynomial's value. Defect is defined in section \ref{rational_minimax}, and it is not really problem which is also explained there.

This paper presents an algorithm for efficient evaluation of Boys functions targeted to modern hardware and especially GPUs where arithmetic is cheap and data movement is expensive. The algorithm is based on rational minimax approximations as well as an asymptotic approximation for large $x$-values. With the parameters presented in this paper, an absolute error $\leq \num{5e-14}$ for Boys functions up to $F_{32}$ is guaranteed.

The paper is organized as follows: Section \ref{theorysection} presents the theoretical foundation of the method, beginning with an overview of rational minimax approximations and Remez' algorithm, followed by a derivation of the Boys function algorithm. Section \ref{benchmarksection} provides benchmark results, and section \ref{conclusionsection} is conclusion. All essential mathematical derivations, as well as the rational minimax coefficients, are provided in the appendices.

\section{Theory}\label{theorysection}
\subsection{Rational minimax approximation}\label{rational_minimax}
A rational minimax approximation of a continuous function $I\overset{f}{\longrightarrow}\mathbb{R}$ defined on a compact interval $I$ is a rational function $r$ of given maximum numerator degree $n$ and maximum denominator degree $m$ that minimizes the maximum error
\begin{equation}\label{minimaxerror}
    \left\|f-r\right\|_\rho := \max_{x\in I}|\rho(x)(f(x)-r(x))|,
\end{equation}
where $I\overset{\rho}{\longrightarrow}\langle0,\infty\rangle$ is a continuous weight function. For finite intervals, rational minimax approximations always exist and are unique~[\onlinecite{trefethen2019approximation}, chap. 24], that is, for every $(f,n,m,\rho)$ one can assign a unique rational minimax approximation $r_{n,m,\rho}(f)$. For compact infinite intervals of the form $I = [a,\infty]$, existence and uniqueness still holds. Existence is proven in Appendix~\ref{existenceofminimaxoninfiniteinterval}. This can be used to generate minimax approximations on the whole non-negative $x$-axis which leads to a branch-free Boys function evaluation algorithm because the asymptotic approximation is not needed. This was attempted in this work, but was unsuccessful because of high polynomial degree which lead to numerical instability. Laikov seems to have succeeded in achieving this goal by representing polynomials in a different way.~\cite{Laikov2025}

Rational minimax approximations are uniquely described by the equioscillation theorem. To understand the theorem, we need the concept of \emph{defect} of a rational function. Denote $\mathcal{R}_{n,m}$ the set of all rational functions with maximum numerator degree $n$ and maximum denominator degree $m$. The defect of a rational function $r\in\mathcal{R}_{n,m}$ is defined as
\begin{equation}
    d_{n,m}(r) := \min\{n-\deg p,m-\deg q\},
\end{equation}
where $r = \frac{p}{q}$ and the polynomials $p$ and $q$ have no common factor of at least degree $1$. Note that if a rational function has a defect in $\mathcal{R}_{n,m}$, the defect can be removed by simply subtracting $n$ and $m$ by the defect. The equioscillation theorem reads:
\begin{theorem}\label{equioscillationtheorem} \mbox{} \\
    Let $r\in\mathcal{R}_{n,m}$ have zero defect. Then \\
    $r$ is a minimax approximation of $f$ 
    $$\Big\Updownarrow$$
    there are $N := n+m+2$ points $x_1<\dots<x_N$ in $I$ and a sign $\sigma\in\{-1,+1\}$ so that
    \begin{equation}\label{equioscillation}
        \rho(x_i)(f(x_i) - r(x_i)) = (-1)^i \sigma \|f-r\|_\rho
    \end{equation}
    for all $i=1,\dots,N$.
\end{theorem}
A proof can be found in [\onlinecite{trefethen2019approximation}, chap. 24]. That proof does not consider $\rho$-weighted error, but generalizing the proof is straightforward. To compute minimax approximations, one can solve eq. (\ref{equioscillation}) with Remez' algorithm.


\subsection{Remez' algorithm}
The rational Remez algorithm iteratively solves the equioscillation equation (\ref{equioscillation}) for the unkown equioscillation nodes $x_1,\dots,x_N\in I$ and rational function $r\in\mathcal{R}_{n,m}$. The algorithm was initially published in 1962, but a modern and robust version is described in [\onlinecite{silviu2018rational}] based on divided differences and a convergence theorem in [\onlinecite{WatsonG.A1980Atan}, theorem 9.14].

The algorithm can be divided into 6 steps.
\begin{description}
    \item[Step 1] Guess equioscillation nodes. According to theorem 9.14 in [\onlinecite{WatsonG.A1980Atan}], if the initial guess of equioscillation sufficiently good, the algorithm will converge as long as the node updates are performed in accordance with point i) and ii) in step 5. Instead of trying to find a sophisticated method of getting a good initial guess it easier and computationally cheap to make a random guess, and if the guess does not result in any pole free-solution in step 2, a new random guess is made.
    \item[Step 2] Solve equioscillation equation with fixed nodes
    \begin{equation}\label{equioscillationequationwithfixednodes}
        \rho\cdot(f-r)(x_n) = (-1)^n E
    \end{equation}
    for $r=\frac{p}{q}\in\mathcal{R}_{n,m}$ and $E\in\mathbb{R}$ using divided differences and two orthonormal polynomial bases to represent $p$ an $q$ respectively. The divided differences separates the unknowns $p$ and $q$ into decoupled equations. The orthonormal basis for $q$ turns eq. (\ref{equioscillationequationwithfixednodes}) into a real symmetrixc eigenvalue problem, resulting in $m+1$ independent solutions $(r_1,E_1),\dots,(r_{m+1},E_{m+1})$.
    \item[Step 3] Check if a pole-free solution $r$ to eq. (\ref{equioscillationequationwithfixednodes}) was found. One can show there is at most one such solution without poles in the interval $I$.~\cite{silviu2018rational} Searching for poles is done by using Sturm's theorem.~\cite{Basu2006} For each solution $r_j\in\{r_1,\dots,r_{m+1}\}$, compute the Sturm sequence of the denominator $q_j$. From the Sturm sequence of $q_j$, one can compute the number of zeros in any interval. Once one finds a $q_j$ without zeros in $I$, step 3 is successful, and one can move to step 4. If all $q_1,\dots,q_{m+1}$ have zeros in $I$, step 3 was unsuccessfull, and one moves back to step 1.
    \item[Step 4] Check for convergence. This is done by picking a threshold $\varepsilon \geq 0$ and testing if
    \begin{equation}
        \|f-r\|_\rho - \left|E^{(i)}\right| \leq \varepsilon.
    \end{equation}
    This convergence criterion is motivated by that if $\lim_{i\rightarrow\infty}\left|E^{(i)}\right| = \|f-r\|_\rho$, then $\lim_{i\rightarrow\infty}r^{(i)} = r^*$, see definition 3.6 and theorem 6.7 in [\onlinecite{Braess1986}]. Furthermore, we will always have $\left|E^{(i)}\right| \leq \|f-r\|_\rho$ by de la Vallée-Poussin's theorem.~\cite{Braess1986}
    \item[Step 5] Update equioscillation nodes $x_n^{(i)} \mapsto x_n^{(i+1)}$ so that
    \begin{enumerate}
        \item[i)] $(-1)^n \sigma \rho\cdot\left(f-r^{(i)}\right)\left(x_n^{(i+1)}\right) \geq \left|E^{(i)}\right|$ for a ${\sigma\in\{-1,1\}}$
        \item[ii)] One of the new nodes is global maximum of ${\rho\cdot\left|f-r^{(i)}\right|}$.
    \end{enumerate}
    This was done by choosing the new nodes to be local maxima of $\rho\cdot\left|f-r^{(i)}\right|$ such that point 1 and 2 holds. Location of local maxima of $\rho\cdot\left|f-r^{(i)}\right|$ was done as follows:
    \begin{enumerate}
        \item Evaluate $\rho\cdot\left|f-r^{(i)}\right|$ on a grid $$\{y_1,\dots,y_K\}\subseteq [a,b]$$ fine enough that we can assume ${\rho\cdot\left|f-r^{(i)}\right|}$ is unimodal between each grid point.
        \item Find local maxima of $\rho\cdot\left|f-r^{(i)}\right|$ on the grid.
        \item For each grid point $y_k\in I$ that is a local maximum, perform golden section search\cite{lu2007practical} on the interval $[y_{k-1},y_{k+1}]$ to find true local maximum, except for end points. As for end points, if $y_k = y_1$, the interval is $[a, y_2]$, and if $y_k = y_K$, the interval is $[y_{K-1}, b]$. Golden section search works if we have unimodality.
    \end{enumerate}
    \item[Step 6] Check if desired error threshold is achievable. Given an error threshold $E_\mathrm{tol} > 0$, abort early if
    \begin{equation}\label{abortearly}
        \min_n \rho\cdot\left|f-r^{(i)}\right|\left(x_n^{(i+1)}\right) > E_\mathrm{tol}.
    \end{equation}
    If inequality (\ref{abortearly}) holds, then de la Vallée-Poussin's theorem~\cite{Braess1986} says it is impossible to reach a minimax error less or equal to $E_\mathrm{tol}$.
\end{description}
The real symmetric matrix in step 2 of the algorithm was diagonalized with a copy of Netlib's LAPACK implementation~\cite{laug} that is available at \url{https://gitlab.com/rasmusvi/qlapack}. The code is compiled with the flag \texttt{-freal-8-real-16} for the GNU compiler \texttt{gfortran} and \texttt{-ipo -double-size 128} for the Intel compiler \texttt{ifx} to turn double precision into quadruple precision. To compile with these flags is a bit risky, but it was tested and seems to work. The flow chart in fig.~\ref{flowchartfigure} demonstrates how the steps are connected.
\begin{widetext}
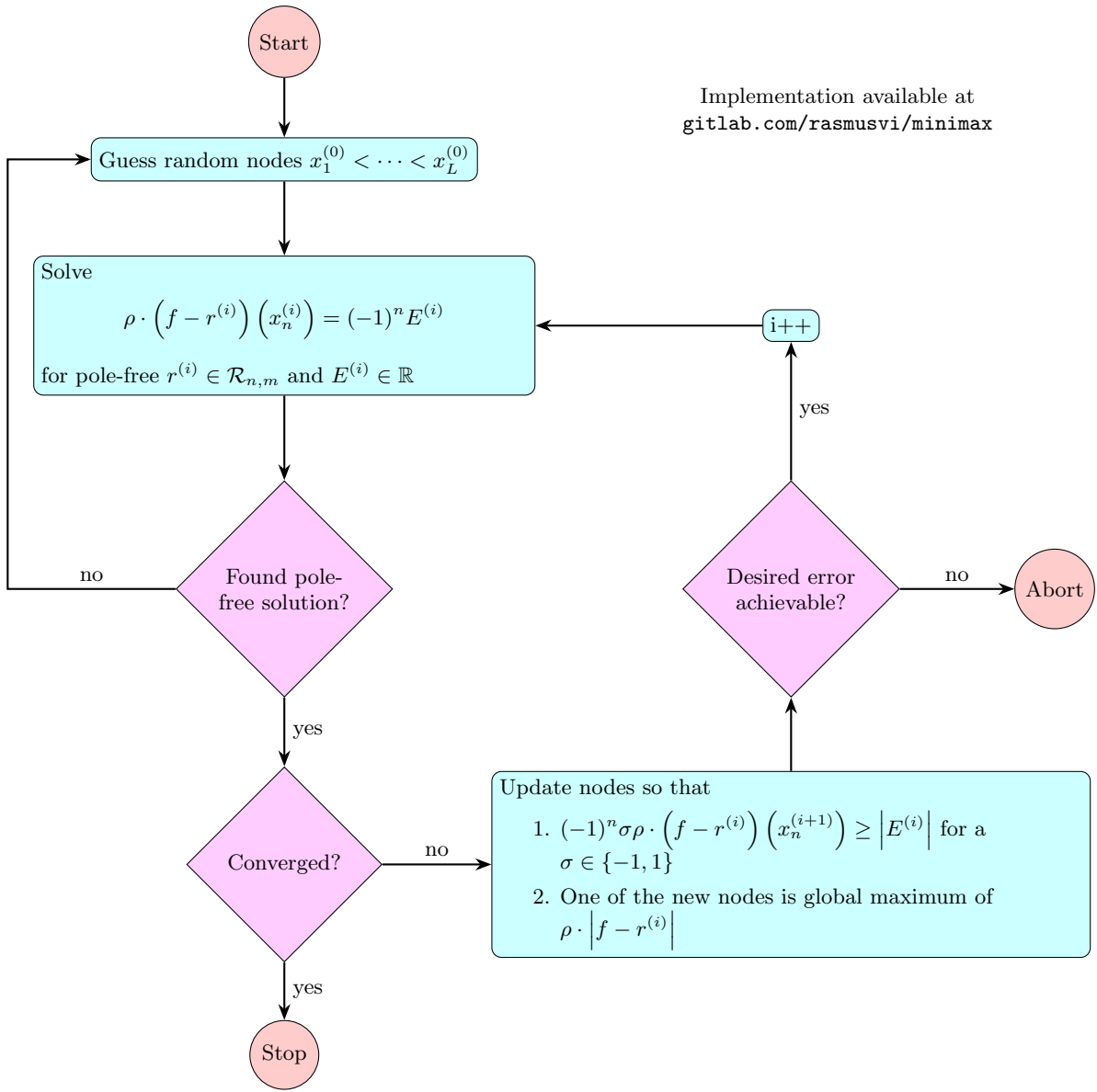
\begin{figure}[H]
\begin{tikzpicture}
    \def\xshift{180}
    
    \node (start)
        [startstop]
        {Start};
    \node (initial guess)
        [basicblock, below of=start, yshift=-20]
        {Guess random nodes $x_1^{(0)}<\dots<x_L^{(0)}$};
    \node (solve)
        [basicblock, below of=initial guess, yshift=-40, text width=200]
        {Solve
        \begin{equation*}
            \rho\cdot\left(f-r^{(i)}\right)\left(x_n^{(i)}\right) = (-1)^n E^{(i)}
        \end{equation*}
        for pole-free $r^{(i)}\in\mathcal{R}_{n,m}$ and $E^{(i)}\in\mathbb{R}$};
    \node (found solution)
        [decision, below of=solve, yshift=-80, text width=60]
        {Found pole-free solution?};
    \node (converged)
        [decision, below of=found solution, yshift=-85, text width=60]
        {Converged?};
    \node (update nodes)
        [basicblock, right of=converged, xshift=\xshift, text width=240]
        {Update nodes so that
        \begin{enumerate}
        \item $(-1)^n \sigma \rho\cdot\left(f-r^{(i)}\right)\left(x_n^{(i+1)}\right) \geq \left|E^{(i)}\right|$ for a $\sigma\in\{-1,1\}$
        \item One of the new nodes is global maximum of $\rho\cdot\left|f-r^{(i)}\right|$
    \end{enumerate}};
    \node (error test)
        [decision, right of=found solution, xshift=\xshift, text width=60]
        {Desired error \\ achievable?};
    \node (abort)
        [startstop, right of=error test, xshift=80]
        {Abort};
    \node (increment)
        [basicblock, right of=solve, xshift=\xshift]
        {i++};
    \node (stop)
        [startstop, below of=converged, yshift=-50]
        {Stop};
    
    \draw[thickarrow] (start) -- (initial guess);
    \draw[thickarrow] (initial guess) -- (solve);
    \draw[thickarrow] (solve) -- (found solution);
    \draw[thickarrow] (found solution) -| node[near start, above] {no} (-4,-4) |- (initial guess);
    \draw[thickarrow] (found solution) -- node[right] {yes} (converged);
    \draw[thickarrow] (converged) -- node[above] {no} (update nodes);
    \draw[thickarrow] (update nodes) -- (error test);
    \draw[thickarrow] (error test) -- node[right] {yes} (increment);
    \draw[thickarrow] (error test) -- node[above] {no} (abort);
    \draw[thickarrow] (increment) -- (solve);
    \draw[thickarrow] (converged) -- node[right] {yes} (stop);
    
    \node (QRcode) at (8,-1) {\shortstack{Implementation available at \\ \url{gitlab.com/rasmusvi/minimax}}};
\end{tikzpicture}
\caption{Flow chart of the implementation Remez' algorithm}
\label{flowchartfigure}
\end{figure}
\end{widetext}

\subsection{Derivation of Boys function algorithm}
Since minimax approximations are designed to work on compact intervals, the non-negative $x$-axis is partitioned into two regions, an asymtotic and a non-asymtotic.
\begin{equation}
    [0,\infty\rangle = \overbrace{[0,x_1\rangle}^{\mathclap{\textnormal{non-asymptotic region}}}\cup\underbrace{[x_1,\infty\rangle}_{\mathclap{\textnormal{asymptotic region}}}.
\end{equation}
In the non-asymptotic region, minimax approximations are used, and in the asymptotic region, the asymptotic approximation [\onlinecite{helgaker2013molecular}, eq. (9.8.9)]
\begin{equation}\label{asymptoticapproximation}
    F_k(x) \approx \frac{\Gamma\left(k+\frac{1}{2}\right)}{2x^{k+\frac{1}{2}}}
\end{equation}
is used, where $\Gamma$ is the gamma function. Since the asymptotic approximation is computationally cheap, the beginning of the asymptotic region $x_1$ is chosen to be as small as possible such that the error in the asymptotic approximation is within the desired error threshold, call it $\varepsilon_\mathrm{tol}$. Since the error in the asymptotic approximation
\begin{equation}
    \left|\frac{\Gamma\left(k+\frac{1}{2}\right)}{2x^{k+\frac{1}{2}}} - F_k(x)\right| = \frac{\Gamma\left(k+\frac{1}{2},x\right)}{2x^{k+\frac{1}{2}}}
\end{equation}
is decreasing with $x$, this can be achieved by solving the equation
\begin{equation}\label{x1}
    \frac{\Gamma\left(k_\mathrm{max}+\frac{1}{2},x_1\right)}{2x_1^{k_\mathrm{max}+\frac{1}{2}}} = \varepsilon_\mathrm{tol}
\end{equation}
for $x_1$. Here $\Gamma(s,x)$ denotes the upper incomplete gamma function. Eq. (\ref{x1}) can be solved by Newton's method.

In integral evaluation with GTO basis, it is commonly needed to compute all Boys functions $F_0,\dots,F_k$ up to some order $k$ at a time. The upwards and downwards recurrence relations\cite{helgaker2013molecular}
\begin{align}
\label{upwardsrecursion}
    F_{k+1}(x) &= \frac{(2k+1)F_k(x) - e^{-x}}{2x} \\
\label{downwardsrecursion}
    F_k(x) &= \frac{2xF_{k+1}(x) + e^{-x}}{2k+1}
\end{align}
are computationally cheap ways to achieve that. Downwards recursion is generally stable. Upwards recursion is unstable for small $x$. The benefit of upwards recursion is that only minimax approximation of $F_0$ is required to evaluate a batch $F_0,\dots,F_k$.

Generating minimax approximations for the non-asymptotic region leads to high degree polynomials with varying signs in the coefficients. Such polynomials require too high precision in the high degree coefficients for precise evaluation. Therefore, likewise with the method of Mazur et al. \cite{Mazur2016}, the non-asymptotic region is partitioned in two,
\begin{equation}\label{regions}
    [0,\infty\rangle = \overbrace{[0,x_0\rangle\cup[x_0,x_1\rangle}^{\mathclap{\textnormal{non-asymptotic region}}}\cup\underbrace{[x_1,\infty\rangle}_{\mathclap{\textnormal{asymptotic region}}},
\end{equation}
to allow for lower degree minimax approximations where varying signs are not an issue. Label
\begin{equation}
    A := [0,x_0\rangle \qquad B := [x_0,x_1\rangle \qquad C := [x_1,\infty\rangle.
\end{equation}

On region $A$, we need minimax approximations $r_{A,0}, \dots, r_{A,k_\mathrm{max}}$ so that for all $0 \leq l \leq k \leq k_\mathrm{max}$
\begin{equation}
    \max_{x\in A} |F_l(x)-y_{A,k,l}(x)| \leq \varepsilon_\mathrm{tol},
\end{equation}
where $y_{A,k,0},\dots,y_{A,k,k}$ are the values obtained from downwards recursion
\begin{equation}
    y_{A,k,l}(x) := \frac{2xy_{A,k,l+1}(x) + e^{-x}}{2l+1}
\end{equation}
starting at $y_{A,k,k} := r_{A,k}$.

On region $B$, we need a single minimax approximation $r_B$ so that for all $0 \leq l \leq k_\mathrm{max}$
\begin{equation}
    \max_{x\in B} |F_l(x)-y_{B,l}(x)| \leq \varepsilon_\mathrm{tol},
\end{equation}
where $y_{B,0},\dots,y_{B,k_\mathrm{max}}$ are the values obtained from upwards recursion
\begin{equation}
    y_{B,l+1}(x) := \frac{(2l+1)y_{B,l}(x) - e^{-x}}{2x}
\end{equation}
starting at $y_{B,0} = r_B$.

To ensure numerical stability of downwards recursion on region $A$ and upwards recursion on region $B$, we pick weight functions $\rho_{A,0}, \dots, \rho_{A,k_\mathrm{max}}$ on $A$ so that $$\max_{x\in A} |\rho_{A,k}\cdot(F_k-r_{A,k})(x)|$$ is minimized and weight function $\rho_B$ on $B$ so that $$\max_{x\in B} |\rho_B\cdot(F_k-r_B)(x)|$$ is minimized, as well as the splitting point $x_0$ in the following way.

The weight function $\rho_{A,k}$ was chosen so that
\begin{gather}
\nonumber
    \max_{x\in A} |\rho_{A,k}\cdot(F_k-r_{A,k})(x)| \leq \varepsilon_\mathrm{tol} \\
    \big\Downarrow \\
\nonumber
    \max_{x\in A} |F_l(x)-y_{A,k,l}(x)| \leq \varepsilon_\mathrm{tol} \quad\forall l=0,\dots,k.
\end{gather}
With round-off errors approximated up to zeroth order in machine epsilon, this is guaranteed by choosing
\begin{equation}
\label{rhoAk}
    \rho_{A,k}(x) := \max_{l=0,\dots,k} \prod_{n=l}^{k-1}\frac{x}{n+\frac{1}{2}},
\end{equation}
where empty product ($l = k$) is as usual defined as $1$. If the error threshold $\varepsilon_\mathrm{tol}$ is close to machine epsilon, one can use the exact formula (\ref{exactrhoAk}). See Appendix \ref{derivation_of_x0_and_rhoAk} for derivation of eq. (\ref{rhoAk}).

The splitting point $x_0$ was chosen to be as small as possible such that upwards recursion is stable on $[x_0,\infty\rangle$ with the choice
\begin{equation}
    \rho_B(x) := 1.
\end{equation}
Up to zeroth order in round-off errors, this results in
\begin{equation}\label{x0}
    x_0 = \max\left\{1, \left(\prod_{k=0}^{k_\mathrm{max}-1}k+\frac{1}{2}\right)^\frac{1}{k_\mathrm{max}}\right\},
\end{equation}
where $k_\mathrm{max}$ is the maximum Boys function order reached by upwards recursion from $F_0$. See Appendix \ref{derivation_of_x0_and_rhoAk} for derivation of eq. (\ref{x0}). Fig. \ref{regions_figure} summarizes the method.

\begin{figure}
    \centering
    \includesvg[scale=0.35]{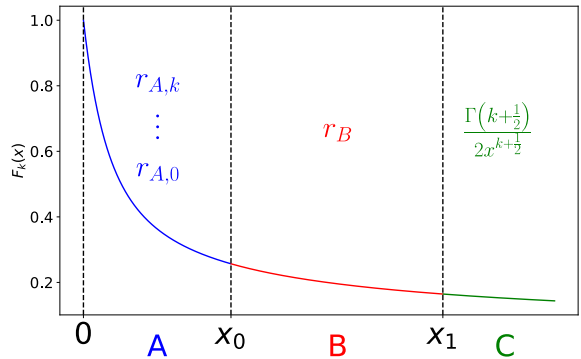}
    \caption{Schematic representation of partitions used for Boys function evaluation.}
    \label{regions_figure}
\end{figure}

\subsection{Choice of minimax from the Walsh table}\label{walshtable}
The Walsh table of $f$ is the table of all minimax approximations of $f$ with maximum numererator degree on the vertical axis and maximum denominator degree on the horizontal axis like shown below.
\begin{equation*}
\begin{array}{|c|c|c|c}
    \vdots          &                 &                 & \iddots \\
    \hline
    r_{2,0,\rho}(f) & r_{2,1,\rho}(f) & r_{2,2,\rho}(f) & \\
    \hline
    r_{1,0,\rho}(f) & r_{1,1,\rho}(f) & r_{1,2,\rho}(f) & \\
    \hline
    r_{0,0,\rho}(f) & r_{0,1,\rho}(f) & r_{0,2,\rho}(f) & \dots\\
    \hline
\end{array}
\end{equation*}

What we want to minimize is the number arithmetic operations given an error threshold $\varepsilon_\mathrm{tol}$. The number of arithmetic operations in the evaluation of a rational function is $2(n+m)$, where $n$ is the numerator degree and $m$ is denominator degree. Given an error threshold, if $r$ is a minimax approximation with minimal $n+m$ among all minimax approximations on the Walsh table with error within the threshold, then $r$ minimizes number of arithmetic operations under the constraint of being within the error threshold.

Therefore, the minimax approximations $r_{A,0},\dots,r_{A,k_\mathrm{max}}$ and $r_B$ were chosen in that manner. This was done by traversing along consecutive anti-diagonals of the Walsh table as shown below
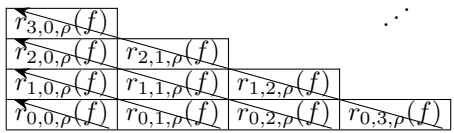
\begin{equation*}
\begin{array}{ccccc}
    \cline{1-1}
    \multicolumn{1}{|c|}{\tikzmarknode{b3}{r_{3,0,\rho}(f)}}  &                                                         &                                                         & \smash{\iddots} \\
    \cline{1-2}
    \multicolumn{1}{|c|}{\tikzmarknode{b2}{r_{2,0,\rho}(f)}}  & \multicolumn{1}{c|}{r_{2,1,\rho}(f)}                    &                                                         & \\
    \cline{1-3}
    \multicolumn{1}{|c|}{\tikzmarknode{b1}{r_{1,0,\rho}(f)}}  & \multicolumn{1}{c|}{r_{1,1,\rho}(f)}                    & \multicolumn{1}{c|}{r_{1,2,\rho}(f)}                    & \\
    \cline{1-4}
    \multicolumn{1}{|c|}{\tikzmarknode{ba0}{r_{0,0,\rho}(f)}} & \multicolumn{1}{c|}{\tikzmarknode{a1}{r_{0,1,\rho}(f)}} & \multicolumn{1}{c|}{\tikzmarknode{a2}{r_{0,2,\rho}(f)}} & \multicolumn{1}{c|}{\tikzmarknode{a3}{r_{0,3,\rho}(f)}} \\
    \cline{1-4}
\end{array}
\begin{tikzpicture}[overlay, remember picture]
    \draw[arrow] (ba0.south east) -- (ba0.north west);
    \draw[arrow] (a1.south east) -- (b1.north west);
    \draw[arrow] (a2.south east) -- (b2.north west);
    \draw[arrow] (a3.south east) -- (b3.north west);
\end{tikzpicture}
\end{equation*}
and then picking the one with smallest error on the first antidiagonal where at least one minimax approximation meet the desired error threshold.

An alternative could be to pick the minimax approximation of small enough error that is nearest the center of the antidiagonal because having $n$ and $m$ closer to being equal gives the evaluation more concurrency in evaluatating the numerator and denominator in parallel. This can for example be utilized by superscalar technology in modern hardware.

\subsection{Reference values of Boys functions}
To run Remez' algorithm, a way of computing reference values of the function one wants to approximate has to be provided. As for Boys functions the following formula\cite{shavitt1963gaussian} was chosen for its numerical stability.
\begin{equation}
    F_k(x) = \frac{e^{-x}}{2} \sum_{l=0}^\infty \frac{x^l}{\prod_{j=0}^l\left(k+j+\frac{1}{2}\right)}.
\end{equation}
It is numerically stable because all terms in the sum have equal sign. To compute it, the sum had to be truncated by an integer $K$ chosen such that the relative truncation error was low enough. One can show that the relative truncation error is bounded as follows.
\begin{equation}\label{errobound}
    \left|\frac{F_k(x) - \frac{e^{-x}}{2} \sum_{l=0}^L \frac{x^l}{\prod_{j=0}^l\left(k+j+\frac{1}{2}\right)}}{F_k(x)}\right| \leq{} \frac{x^{k+L+\frac{3}{2}}}{\Gamma\left(k+L+\frac{3}{2}\right)}.
\end{equation}
$L$ was set to $150$ which gives a relative error $\leq 1.28\cdot10^{-69}$, and consequently the absolute error is also ${\leq 1.28\cdot10^{-69}}$ since $|F_k(x)|\leq 1$. See Appendix~\ref{derivation_error_bound} for derivation of inequality (\ref{errobound}).

\subsection{Boys function algorithm}
The following pseudocode demonstrates the algorithm for evaluating batches of Boys functions.
\begin{algorithm}[H]
\caption{Evaluate Boys functions $F_0,\dots,F_k$}
\begin{algorithmic}
    \Require $k\in\mathbb{N}_0$
    \Require $x\in[0,\infty\rangle$
    \If{$x\in A$}
        \State $F_k = r_{A,k}(x)$
        \For{$l = k-1,\dots,0$}
            \State $F_l = \frac{2xF_{l+1} + e^{-x}}{2l+1}$
        \EndFor
    \ElsIf{$x\in B$}
        \State $F_0 = r_B(x)$
        \For{$l = 0,\dots,k-1$}
            \State $F_{l+1} = \frac{(2l+1)F_l - e^{-x}}{2x}$
        \EndFor
    \Else
        \State $F_0 = \frac{\sqrt{\pi}}{2} \frac{1}{\sqrt{x}}$
        \For{$l = 0,\dots,k-1$}
            \State $F_{l+1} = \frac{2l+1}{2x}F_l$
        \EndFor
    \EndIf
    \State \Return $F_0,\dots,F_k$
\end{algorithmic}
\end{algorithm}

\section{Benchmark}\label{benchmarksection}
It is difficult to compare different algorithms for evaluating Boys functions because all Boys function algorithms have the same scaling with respect to the number of input arguments, namely linear. Which algorithm is faster can depend heavily on the hardware. Therefore, it is \textit{very limited} what conclusions one can draw from a benchmark unless the speed-up is very high. Nevertheless, this method was benchmarked against Ishida's method \cite{ishida1996ace} and a more modern method by Beylkin and Sharma \cite{beylkin2021fast}.

The benchmark is described by Algorithm~\ref{benchmarkalgorithm}. It takes in three randomly initialized vectors $\mathbf{x},\mathbf{y},\mathbf{c}$ and computes an output vector $\mathbf{z}$ given by
\begin{equation}\label{benchmarkequation}
    z_i = \sum_{l=0}^kc_l\sum_jF_l(x_i+x_j)y_j.
\end{equation}
The Boys functions are evaluated on $N^2$ input values of the form $X_{ij} = x_i+x_j$ instead of just the vector $\mathbf{x}$ of length $N$ to avoid having the benchmark being bound by memory traffic in case the hardware has high latency relative to its maximal throughput. In the benchmark, the maximum Boys function order $k$ had to be set equal to 12 since Beylkin and Sharma only provided coefficients to evaluate Boys function $F_0$ and $F_{12}$.

The benchmark code was compiled with the Nvidia compiler nvc version 24.9-0 using OpenACC for parallelization and run on a 32 core Intel Xeon Gold 6338 CPU and an Nvidia A100 GPU. The benchmark evaluated the Boys functions on $N$ points randomly sampled from the interval $[0,30]$ with $N = 2^{14}$ on the Intel Xeon Gold and $N = 2^{19}$ on the Nvidia A100. The interval was chosen to be $[0,30]$ because all methods can be replaced by the asymptotic approximation in the asymptotic region, and it is therefore not interesting to perform a benchmark for $x$-values in that region. The results of the benchmark are shown in Fig.~\ref{CPUbenchmark} and \ref{GPUbenchmark}. The benchmark code is included as supplementary material.

\begin{algorithm}[H]
\caption{Boys function benchmark}\label{benchmarkalgorithm}
\begin{algorithmic}
    \Require $N,k\in\mathbb{N}_0$
    \Require $\mathbf{x}\in[0,\infty\rangle^N$
    \Require $\mathbf{y}\in\mathbb{R}^N$
    \Require $\mathbf{c}\in\mathbb{R}^{k+1}$
    \For{$i=0,\dots,N-1$}
        \State $z_i = 0$
        \For{$j=0,\dots,N-1$}
            \State $x = x_i+x_j$
            \If{downwards recursion for $x$}
                \State $w = c_k F_k(x)$
                \For{$l=k-1,\dots,0$}
                    \State $F_l(x) = \frac{xF_{l+1}(x) + \frac{1}{2}e^{-x}}{l+\frac{1}{2}}$
                    \State $w \mathrel{{+}{=}} c_l F_l(x)$
                \EndFor
            \ElsIf{upwards recursion for $x$}
                \State $w = c_0 F_0(x)$
                \For{$l=1,\dots,k$}
                    \State $F_l(x) = \frac{\left(l-\frac{1}{2}\right)F_{l-1}(x) - \frac{1}{2}e^{-x}}{x}$
                    \State $w \mathrel{{+}{=}} c_l F_l(x)$
                \EndFor
            \EndIf
        \State $z_i \mathrel{{+}{=}} y_j w$
        \EndFor
    \EndFor
    \State \Return $\mathbf{z}$
\end{algorithmic}
\end{algorithm}
 
\begin{figure}
\begin{tikzpicture}[scale=0.9]
\begin{axis}[
    ybar,
    xticklabels={minimax, Beylkin\&Sharma, Ishida},
    xtick={1,2,3},
    ymin=0,
	ylabel=Time (s),
    legend entries={32 core Intel Xeon Gold 6338 CPU 2.00 GHz},
    legend style={at={(0,1)}, anchor=south west, font=\tiny}
    ]
    \addplot coordinates{(1,40.4) (2,81.2) (3,58.5)};
\end{axis}
\end{tikzpicture}
\caption{Benchmark result of Algorithm \ref{benchmarkalgorithm} with $N^2 = \left(2^{14}\right)^2$ random points in $[0,30]$ on 32 core Intel Xeon Gold 6338 CPU 2.00 GHz, "minimax" is the new scheme presented in this paper, Beylkin\&Sharma is the method in [\onlinecite{beylkin2021fast}] and Ishida is the method in [\onlinecite{ishida1996ace}].}
\label{CPUbenchmark}
\end{figure}
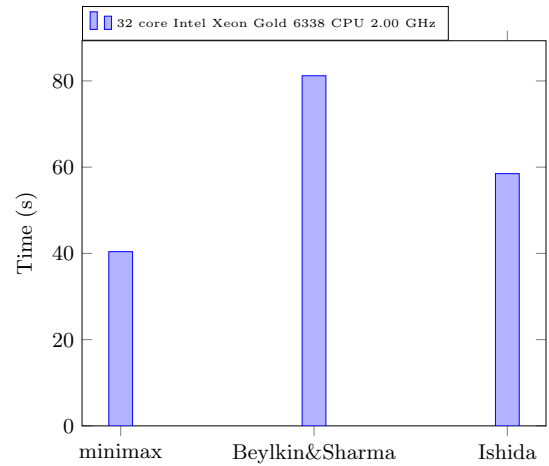

\begin{figure}
\begin{tikzpicture}[scale=0.9]
\begin{axis}[
    ybar,
    xticklabels={minimax, Beylkin\&Sharma, Ishida},
    xtick={1,2,3},
    ymin=0,
	ylabel=Time (s),
    legend entries={Nvidia A100},
    legend style={at={(0,1)}, anchor=south west, font=\tiny}
    ]
    \addplot coordinates{(1,72) (2,87.7) (3,226.7)};
\end{axis}
\end{tikzpicture}
\caption{Benchmark result of Algorithm \ref{benchmarkalgorithm} with $N^2 = \left(2^{19}\right)^2$ random points in $[0,30]$ on Nvidia A100 GPU, "minimax" is the new scheme presented in this paper, Beylkin\&Sharma is the method in [\onlinecite{beylkin2021fast}] and Ishida is the method in [\onlinecite{ishida1996ace}].}
\label{GPUbenchmark}
\end{figure}
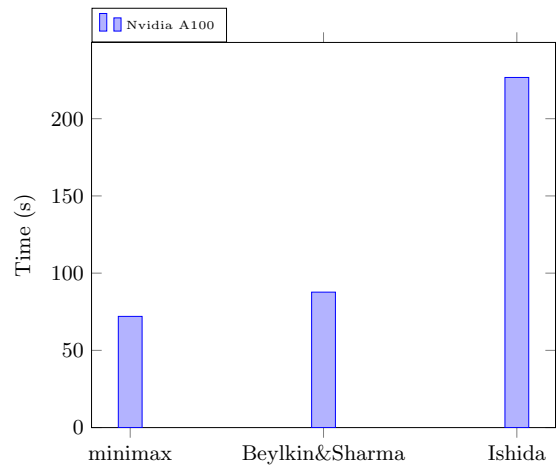

\section{Conlusion}\label{conclusionsection}
In this work, we have presented an efficient algorithm for evaluation of Boys functions $F_{0},\dots,F_{k_\textnormal{max}}$, designed specifically for modern computing architectures where maximum throughput is high and data movement is often a performance bottleneck. The method combines rational minimax approximations with controlled upwards and downwards recursion and an asymptotic treatment for large arguments. By partitioning the non-negative real axis into three regions and applying appropriate approximation strategies in each, the algorithm achieves near machine-precision absolute accuracy while maintaining favorable computational characteristics. Using rational minimax coefficients generated via the rational Remez algorithm, we demonstrate that, for a target tolerance of $\varepsilon_\mathrm{tol} = \num{5e-14}$, the Boys functions up to $F_{32}$ can be evaluated with guaranteed maximum absolute error. All coefficients and region parameters required for practical implementation are provided in the appendices. The proposed approach thus offers a robust, accurate, and GPU-friendly alternative to traditional table-based schemes for Boys function evaluation in large-scale quantum chemical calculations.


\section{Supplementary material}
See the supplementary material for minimax coefficients as new-line-separated list and the source code of the benchmark.

\section{Acknowledgment}
This work was supported by the Research Council of Norway through its Centres
of Excellence scheme (Project No. 262695) and its FRIPRO grant (No. 315822). 
For the computational and data storage resources, we acknowledge the support of 
Sigma2---the National Infrastructure for High Performance Computing and Data Storage in Norway, Grant No.~NN14654K. In addition, the authors acknowledge funding form the Slovak Research and Development Agency (grant No. APVV-22-0488), VEGA No. 1/0670/24, and the EU NextGenerationEU through the Recovery and Resilience Plan for Slovakia under the project Nos. 09I05-03-V02-00034 and 09I01-03-V04-00096.

\appendix

\section{Derivation of error bound of reference values of Boys functions}\label{derivation_error_bound}
Ineq. (\ref{errobound}) can be derived as follows.
\begin{align}
\nonumber
    &\left|\frac{F_n(x) - \frac{e^{-x}}{2} \sum_{k=0}^K \frac{x^k}{\prod_{j=0}^k\left(n+j+\frac{1}{2}\right)}}{F_n(x)}\right| \\
    ={}& \frac{\frac{e^{-x}}{2} \sum_{k=K+1}^\infty \frac{x^k}{\prod_{j=0}^k\left(n+j+\frac{1}{2}\right)}}{F_n(x)} \\
    ={}& \frac{\frac{x^{K+1}}{\prod_{j=0}^K\left(n+j+\frac{1}{2}\right)} F_{n+K+1}(x)}{F_n(x)} \\
    \leq{}& \frac{\frac{x^{K+1}}{\prod_{j=0}^K\left(n+j+\frac{1}{2}\right)} \frac{1}{2(n+K+1)+1}}{\frac{\Gamma\left(n+\frac{1}{2}\right)}{2x^{n+\frac{1}{2}}}} \\
    ={}& \frac{x^{n+K+\frac{3}{2}}}{\Gamma\left(n+K+\frac{3}{2}\right)}.
\end{align}

\section{Derivation of $x_0$ and $\rho_{A,k}$}\label{derivation_of_x0_and_rhoAk}
In these two derivations we need to compute upper bounds to round-off errors in floating point arithmetic as well as errors in evaluation of special functions. For any binary floating point operation $$*_\mathrm{float}\in\{+_\mathrm{float},\cdot_\mathrm{float},-_\mathrm{float},/_\mathrm{float}\}$$ corresponding to an exact binary operation $$*\in\{+,\cdot,-,/\}$$ and two floating point numbers $a$ and $b$ we have 
\begin{equation}\label{floatingpointroundofferror}
    a*_\mathrm{float}b = (a*b)\cdot(1+\delta)
\end{equation}
for a $\delta\in[-\varepsilon_\mathrm{mach},\varepsilon_\mathrm{mach}]$, where $\varepsilon_\mathrm{mach}$ is machine epsilon. Furthermore, for any floating point number $x$ we have
\begin{equation}
\label{exponentialroundofferror}
    \widetilde{\exp}(x) = e^{x}\cdot(1+\gamma)
\end{equation}
for a $\gamma\in[-\varepsilon_\mathrm{exp},\varepsilon_\mathrm{exp}]$, where the tilda means the numerically evaluated value of the function.

Let $\Tilde{F}_l(x)$ be the numerically evaluated value of Boys function $F_l(x)$ at $x\in[0,\infty\rangle$. Define the error
\begin{equation}
    \Delta F_l(x) := F_l(x) - \Tilde{F}_l(x).
\end{equation}

\subsection{Derivation of $x_0$}
We consider upwards recursion of Boys functions from $F_0$ to $F_{k_\mathrm{max}}$. Let $x\in B$. Using eq. (\ref{floatingpointroundofferror}) and (\ref{exponentialroundofferror}), we have for $l=0,\dots,k_\mathrm{max}-1$
\begin{align*}
    &\Tilde{F}_{l+1}(x) \\
    ={}& \left((2l+1) \cdot_\mathrm{float} \Tilde{F}_l(x) -_\mathrm{float} \widetilde{\exp}(x)\right)/_\mathrm{float}(2x) \\
    ={}& \frac{(2l+1)\Tilde{F}_l(x)(1+\delta_1) - e^{-x}(1+\gamma)}{2x}(1+\delta_2)(1+\delta_3) \\
    ={}& \underbrace{\frac{(2l+1)F_l(x) - e^{-x}}{2x}}_{= F_{l+1}(x)} - \frac{(2l+1)\Delta F_l(x)}{2x} \\
    &+ \frac{(2l+1)\Tilde{F}_l(x)}{2x}((1+\delta_1)(1+\delta_2)(1+\delta_3)-1) \\
    &- \frac{(2l+1)e^{-x}}{2x}((1+\eta)(1+\delta_2)(1+\delta_3)-1).
\end{align*}
Move $F_{l+1}(x)$ to the other side. One gets
\begin{multline*}
    \Delta F_{l+1}(x) = \frac{l+\frac{1}{2}}{x} \Delta F_l(x) \\
    - \frac{(2l+1)\Tilde{F}_l(x)}{2x}((1+\delta_1)(1+\delta_2)(1+\delta_3)-1) \\
    + \frac{(2l+1)e^{-x}}{2x}((1+\eta)(1+\delta_2)(1+\delta_3)-1).
\end{multline*}
Since $|\delta_i| \leq \varepsilon_\mathrm{mach}$ and $|\gamma| \leq \varepsilon_\mathrm{exp}$, where $\varepsilon_\mathrm{mach}$ is machine epsilon and $\varepsilon_\mathrm{exp}$ is maximum relative round-off error in evaluation of the exponential function, we get
\begin{equation}
    |\Delta F_{l+1}(x)| \leq \frac{l+\frac{1}{2}}{x} |\Delta F_l(x)| + \delta F_{l+1}^\mathrm{up}(x),
\end{equation}
where
\begin{multline}
    \delta F_{l+1}^\mathrm{up}(x) := \frac{(2l+1)\Tilde{F}_l(x)}{2x}\left((1+\varepsilon_\mathrm{mach})^3-1\right) \\
    \frac{(2l+1)e^{-x}}{2x}\left((1+\varepsilon_\mathrm{exp})(1+\varepsilon_\mathrm{mach})^2-1\right).
\end{multline}
Integrating the recursive inequality (\ref{downwardsrecursiveinequality}) over $l$ and using the requirement
\begin{equation*}
    |\Delta F_0(x)| \leq \varepsilon_\mathrm{tol},
\end{equation*}
one gets
\begin{equation*}
    |\Delta F_l(x)| \leq \left(\prod_{n=0}^{l-1}\frac{n+\frac{1}{2}}{x}\right) \varepsilon_\mathrm{tol} + G_l^\mathrm{up}(x),
\end{equation*}
where
\begin{equation}
    G_l^\mathrm{up}(x) := \sum_{m=1}^l \delta F_m^\mathrm{up}(x) \prod_{n=m}^{l-1}\frac{n+\frac{1}{2}}{x}.
\end{equation}
We achieve the desired error bound
\begin{equation*}
    |\Delta F_l(x)| \leq \varepsilon_\mathrm{tol}
\end{equation*}
if we require
\begin{equation*}
    \left(\prod_{n=0}^{l-1}\frac{n+\frac{1}{2}}{x}\right) \varepsilon_\mathrm{tol} + G_l^\mathrm{up}(x) \leq \varepsilon_\mathrm{tol}
\end{equation*}
for all $l = 0,\dots,k_\mathrm{max}$ and for all $x\in[x_0,\infty\rangle$, which is equivalent to
\begin{align*}
    x &\geq \left(\frac{\varepsilon_\mathrm{tol} \prod_{n=0}^{l-1}n+\frac{1}{2}}{\varepsilon_\mathrm{tol} - G_l^\mathrm{up}(x)}\right)^\frac{1}{l} \\
    &= \left(\prod_{n=0}^{l-1}n+\frac{1}{2}\right)^\frac{1}{l} + O(\varepsilon_\mathrm{tol}) + O(\varepsilon_\mathrm{exp}).
\end{align*}
Up to zeroth order in round-off errors, this is achieved by defining
\begin{align}
\nonumber
    x_0 :={}& \max_{l=0,\dots,k_\mathrm{max}} \left(\prod_{n=0}^{l-1}n+\frac{1}{2}\right)^\frac{1}{l} \\
    ={}& \max\left\{1, \left(\prod_{n=0}^{k_\mathrm{max}-1}n+\frac{1}{2}\right)^\frac{1}{k_\mathrm{max}}\right\}.
\end{align}

\subsection{Derivation of $\rho_{A,k}$}
Let $F_k$ be the Boys function from which we will do downwards recursion to evaluate $F_{k-1}, \dots, F_0$. Let $x\in A$. Using eq. (\ref{floatingpointroundofferror}) and (\ref{exponentialroundofferror}), we have for $l=0,\dots,k-1$
\begin{align*}
    &\Tilde{F}_l(x) \\
    ={}& \left(2x \cdot_\mathrm{float} \Tilde{F}_{l+1}(x) +_\mathrm{float} \widetilde{\exp}(x)\right)/_\mathrm{float}(2l+1) \\
    ={}& \frac{\left(\left(2x\Tilde{F}_{l+1}(x)\right)(1+\delta_1) + e^{-x}(1+\gamma)\right)(1+\delta_2)}{2l+1}(1+\delta_3) \\
    ={}& \underbrace{\frac{2xF_{l+1}(x) + e^{-x}}{2l+1}}_{= F_l(x)} - \frac{2x\Delta F_{l+1}(x)}{2l+1} \\
    &+ \frac{2x\Tilde{F}_{l+1}(x)}{2l+1}((1+\delta_1)(1+\delta_2)(1+\delta_3)-1) \\
    &+ \frac{e^{-x}}{2l+1}((1+\gamma)(1+\delta_2)(1+\delta_3)-1).
\end{align*}
Move $F_l(x)$ to the other side. One gets
\begin{multline*}
    \Delta F_l(x) = \frac{x}{l+\frac{1}{2}} \Delta F_{l+1}(x) \\
    - \frac{2x\Tilde{F}_{l+1}(x)}{2l+1}((1+\delta_1)(1+\delta_2)(1+\delta_3)-1) \\
    - \frac{e^{-x}}{2l+1}((1+\gamma)(1+\delta_2)(1+\delta_3)-1).
\end{multline*}
Since $|\delta_i| \leq \varepsilon_\mathrm{mach}$ and $|\gamma| \leq \varepsilon_\mathrm{exp}$, where $\varepsilon_\mathrm{mach}$ is machine epsilon and $\varepsilon_\mathrm{exp}$ is maximum relative round-off error in evaluation of the exponential function, we get
\begin{equation}\label{downwardsrecursiveinequality}
    |\Delta F_l(x)| \leq \frac{x}{l+\frac{1}{2}} |\Delta F_{l+1}(x)| + \delta F_l^\mathrm{down}(x),
\end{equation}
where
\begin{multline}
    \delta F_l^\mathrm{down}(x) := \frac{2x\Tilde{F}_{l+1}(x)}{2l+1}\left((1+\varepsilon_\mathrm{mach})^3-1\right) \\
    + \frac{e^{-x}}{2l+1}\left((1+\varepsilon_\mathrm{exp})(1+\varepsilon_\mathrm{mach})^2-1\right).
\end{multline}
Integrating the recursive inequality (\ref{downwardsrecursiveinequality}) over $l$ and using the requirement
\begin{equation*}
    \rho_{A,k}(x) |\Delta F_k(x)| \leq \varepsilon_\mathrm{tol},
\end{equation*}
one gets
\begin{equation*}
    |\Delta F_l(x)| \leq \left(\prod_{n=l}^{k-1}\frac{x}{n+\frac{1}{2}}\right)\frac{\varepsilon_\mathrm{tol}}{\rho_{A,k}(x)} + G_{l,k}^\mathrm{down}(x)
\end{equation*}
for all $l=0,\dots,k-1$, where
\begin{equation}
    G_{l,k}^\mathrm{down}(x) := \sum_{m=l}^{k-1} \delta F_m^\mathrm{down}(x) \prod_{n=l}^{m-1}\frac{x}{n+\frac{1}{2}}.
\end{equation}
We achieve the desired error bound
\begin{equation*}
    |\Delta F_l(x)| \leq \varepsilon_\mathrm{tol}
\end{equation*}
if we require
\begin{equation*}
    \left(\prod_{n=l}^{k-1}\frac{x}{n+\frac{1}{2}}\right)\frac{\varepsilon_\mathrm{tol}}{\rho_{A,k}(x)} + G_{l,k}^\mathrm{down}(x) \leq \varepsilon_\mathrm{tol}
\end{equation*}
for all $l=0,\dots,k$, which is equivalent to
\begin{align*}
    \rho_{A,k}(x) &\geq \frac{\varepsilon_\mathrm{tol} \prod_{n=l}^{k-1}\frac{x}{n+\frac{1}{2}}}{\varepsilon_\mathrm{tol} - G_{l,k}^\mathrm{down}(x)} \\
    &= \prod_{n=l}^{k-1}\frac{x}{n+\frac{1}{2}} + O(\varepsilon_\mathrm{mach}) + O(\varepsilon_\mathrm{exp}).
\end{align*}
This is achieved by defining
\begin{align}
\label{exactrhoAk}
    \rho_{A,k}(x) :={}& \max_{l=0,\dots,k} \frac{\varepsilon_\mathrm{tol} \prod_{n=l}^{k-1}\frac{x}{n+\frac{1}{2}}}{\varepsilon_\mathrm{tol} - G_{l,k}^\mathrm{down}(x)} \\
    ={}& \max_{l=0,\dots,k} \prod_{n=l}^{k-1}\frac{x}{n+\frac{1}{2}} + O(\varepsilon_\mathrm{mach}) + O(\varepsilon_\mathrm{exp}).
\end{align}

\section{Proof of existence of minimax approximation for infinite interval}\label{existenceofminimaxoninfiniteinterval}
For finite interval, proof for existence of rational minimax approximation can be found in [\onlinecite{trefethen2019approximation}, chap. 24]. Following is a theorem with proof for the case of infinite interval.
\begin{theorem}[Existence of minimax on infinite interval] \mbox{} \\
    Let $[0,\infty]\overset{f}{\longrightarrow}\mathbb{R}$ and $[0,\infty]\overset{\rho}{\longrightarrow}\langle0,\infty\rangle$ be continuous, $n,m\in\mathbb{N}_0$ and $n\leq m$. Then $f$ has a rational minimax approximation in $\mathcal{R}_{n,m}$ with error weighted by $\rho$.
\end{theorem}
\begin{proof} \mbox{} \\
    The idea of the proof is to find a compact set where all minimax approximations must be and use the continuity of the supremum norm to show that there is a minimum in the compact set, which then must be a minimax approximation since the compact set contains all of them. Denote $\mathcal{P}_k$ to be the vector space of real polynomials of maximum degree $k$. A rational function in $\mathcal{R}_{n,m}$ can be represented as an element of $\mathcal{P}_n\times\mathcal{P}_m$. Let $r\in\mathcal{R}_{n,m}$ be a minimax approximation. Let $\|\cdot\|_\rho$ denote the supremum norm defined in eq. (\ref{minimaxerror}). Then by the reverse triangle inequality
    \begin{equation*}
        \|r\|_\rho \leq 2\|f\|_\rho,
    \end{equation*}
    since otherwise
    \begin{align*}
        |\|f\|_\rho - \|r\|_\rho| &\leq \|f-r\|_\rho \\
        \|f\|_\rho &< \|f-r\|_\rho
    \end{align*}
    which contradicts the optimality of $r$. Here, $\|\cdot\|_\rho$ is the $\rho$-weighted supremum norm defined in eq. (\ref{minimaxerror}). To take supremum norm of polynomials, define the damping function
    \begin{equation*}
        d(x) := e^{-x}.
    \end{equation*}
    $r = \frac{p}{q}$ for a $p\in\mathcal{P}_n$ and a $q\in\mathcal{P}_m$, and we can choose $q$ to satisfy
    \begin{equation*}
        \|dq\|_\rho = 1.
    \end{equation*}
    Then
    \begin{align*}
        \left\|\frac{p}{q}\right\|_\rho &\leq 2\|f\|_\rho \\
        \left\|\frac{dp}{dq}\right\|_\rho &\leq 2\|f\|_\rho \\
        \frac{\|dp\|_\rho}{\|dq\|_\rho} &\leq 2\|f\|_\rho \\
        \|dp\|_\rho &\leq 2\|f\|_\rho,
    \end{align*}
    which means all minimax approximations $\frac{p}{q}$ can be represented as an element in the the set
    \begin{equation*}
        K := \{(p,q)\in\mathcal{P}_n\times \mathcal{P}_m \ | \ \|dp\|_\rho \leq 2\|f\|_\rho, \|dq\|_\rho = 1\}.
    \end{equation*}
    $K$ is compact in the norm topology of $\mathcal{P}_n\times\mathcal{P}_m$ since it is closed and bounded, and all norms on a finite-dimensional vector space are equivalent. The map
    \begin{align*}
        K&\longrightarrow[0,\infty] \\
        (p,q)&\mapsto\left\|f-\frac{p}{q}\right\|_\rho
    \end{align*}
    is continuous and hence a minimax approximation exists.
\end{proof}

\bibliography{bibliography}

\section{Coefficients of minimax approximations of Boys functions}\label{coefficients_and_regions}
For an error threshold of \num{5e-14}, the separation points between region $A$, $B$ and $C$ defined in eq. (\ref{x0}) and (\ref{x1}) are given in eq. (\ref{value_of_x0}) and (\ref{vaklue_of_x1}), and the coefficients of the rational minimax approximations are shown in table \ref{r_B_coeffs} - \ref{r_A_32_coeffs}. The coefficients are ordered with degree increasing downwards. The coefficients are meant to be stored in double precision format. Therefore they are given with 17 significant decimal digits which is enough to get the correctly rounded double precision numbers.
\begin{align}
\label{value_of_x0}
x_0 &= 11.899848152108484 \\
\label{vaklue_of_x1}
x_1 &= 28.989337738820740.
\end{align}
\begin{table}[H]
\caption{Rational minimax coefficients for $F_0$ on $[x_0,x_1\rangle$}
\label{r_B_coeffs}
\begin{ruledtabular}
\begin{tabular}{S[table-format=+1.16e+2]S[table-format=+1.16e+2]}
{Numerator} & {Denominator} \\
\hline
 5.74537531702047552E+07 &  4.79893571439451030E+07 \\
 2.73330925890901898E+06 &  3.04808499107506708E+07 \\
 7.52922255805293133E+04 & -1.66693114610725015E+06 \\
 2.33846894861346960E+05 &  5.63505368535215625E+05 \\
 8.34841284469484906E+03 &  6.39702496081641495E+04 \\
 3.90892739018191431E+01 &  8.53693546919731980E+02 \\
  & 1.00000000000000000E+00
\end{tabular}
\end{ruledtabular}
\end{table}
\begin{table}[H]
\caption{Rational minimax coefficients for $F_{0}$ on $[0,x_0\rangle$}
\label{r_A_0_coeffs}
\begin{ruledtabular}
\begin{tabular}{S[table-format=+1.16e+2]S[table-format=+1.16e+2]}
{Numerator} & {Denominator} \\
\hline
 4.59649054199586751E+11 &  4.59649054199579770E+11 \\
 7.24610171100856232E+10 &  2.25677368510488844E+11 \\
 2.24977231104248461E+10 &  5.17586071870896154E+10 \\
 1.62899741137514774E+09 &  7.25815475661893057E+09 \\
 1.91702978974343428E+08 &  6.80492889773299134E+08 \\
 6.56389165108291995E+06 &  4.33436553747085297E+07 \\
 3.22527508970295511E+05 &  1.77090545597099048E+06 \\
 &  3.59362735209789862E+04 \\
 & -2.11809634725166180E+02 \\
 &  1.00000000000000000E+00   
\end{tabular}
\end{ruledtabular}
\end{table}
\begin{table}[H]
\caption{Rational minimax coefficients for $F_{1}$ on $[0,x_0\rangle$}
\label{r_A_1_coeffs}
\begin{ruledtabular}
\begin{tabular}{S[table-format=+1.16e+2]S[table-format=+1.16e+2]}
{Numerator} & {Denominator} \\
\hline
-4.65157653173317170E+11 & -1.39547295952001886E+12 \\
 2.29425320006178902E+10 & -7.68456179705352370E+11 \\
-1.29857712372204999E+10 & -2.01001101693493424E+11 \\
 1.78844602696723749E+08 & -3.29212211155791438E+10 \\
-7.40895343861489278E+07 & -3.73769996691396548E+09 \\
-8.83205562809530090E+04 & -3.06209737336929359E+08 \\
-9.88905350089899030E+04 & -1.81100571529961951E+07 \\
 & -7.35772618617600437E+05 \\
 & -1.73370711526267371E+04 \\
 & -3.01644420112301709E+01 \\
 &  1.00000000000000000E+00   
\end{tabular}
\end{ruledtabular}
\end{table}
\begin{table}[H]
\caption{Rational minimax coefficients for $F_{2}$ on $[0,x_0\rangle$}
\label{r_A_2_coeffs}
\begin{ruledtabular}
\begin{tabular}{S[table-format=+1.16e+2]S[table-format=+1.16e+2]}
{Numerator} & {Denominator} \\
\hline
-3.21534353039794617E+11 & -1.60767176519862195E+12 \\
 3.54003705695069524E+10 & -9.71335122319310198E+11 \\
-7.01635383055901375E+09 & -2.82317080188832696E+11 \\
 2.79078906677022317E+08 & -5.22376678369643476E+10 \\
-2.78149387526899752E+07 & -6.85230379697842854E+09 \\
 2.50304977467284799E+05 & -6.69937093077230221E+08 \\
-2.61245797045770042E+04 & -4.97459753170508105E+07 \\
 & -2.78418670876310491E+06 \\
 & -1.11322872042882201E+05 \\
 & -2.82622020347674619E+03 \\
 &  1.00000000000000000E+00   
\end{tabular}
\end{ruledtabular}
\end{table}
\begin{table}[H]
\caption{Rational minimax coefficients for $F_{3}$ on $[0,x_0\rangle$}
\label{r_A_3_coeffs}
\begin{ruledtabular}
\begin{tabular}{S[table-format=+1.16e+2]S[table-format=+1.16e+2]}
{Numerator} & {Denominator} \\
\hline
 1.74242490762361812E+12 &  1.21969743533675925E+13 \\
-2.31739940313066278E+11 &  7.86435602580557761E+12 \\
 3.01841796858242589E+10 &  2.44715513739141483E+12 \\
-1.26158491469042445E+09 &  4.86816942362079024E+11 \\
 6.24293777436041829E+07 &  6.90443144069857960E+10 \\
 &  7.35560133650822347E+09 \\
 &  6.01952300371660024E+08 \\
 &  3.77439507311669504E+07 \\
 &  1.75556856301654242E+06 \\
 &  5.15451071696163987E+04 \\
 &  7.51530879218449388E+02 \\
 & -6.90725843407910436E+01 \\
 &  1.00000000000000000E+00   
\end{tabular}
\end{ruledtabular}
\end{table}
\begin{table}[H]
\caption{Rational minimax coefficients for $F_{4}$ on $[0,x_0\rangle$}
\label{r_A_4_coeffs}
\begin{ruledtabular}
\begin{tabular}{S[table-format=+1.16e+2]S[table-format=+1.16e+2]}
{Numerator} & {Denominator} \\
\hline
 1.07640450297782221E+08 &  9.68764052679927282E+08 \\
-2.55458353459970974E+07 &  5.62712615905041359E+08 \\
 3.52812080853076687E+06 &  1.56812915525206930E+08 \\
-2.95422780643995305E+05 &  2.77339407473815158E+07 \\
 1.70289049867735511E+04 &  3.46473761130405637E+06 \\
-6.62342704137285632E+02 &  3.21197484344255914E+05 \\
 1.70870241353972667E+01 &  2.24921989775218534E+04 \\
-2.64791558310288340E-01 &  1.17843724039343395E+03 \\
 1.88027971847219425E-03 &  4.36972585953520361E+01 \\
 &  1.00000000000000000E+00   
\end{tabular}
\end{ruledtabular}
\end{table}
\begin{table}[H]
\caption{Rational minimax coefficients for $F_{5}$ on $[0,x_0\rangle$}
\label{r_A_5_coeffs}
\begin{ruledtabular}
\begin{tabular}{S[table-format=+1.16e+2]S[table-format=+1.16e+2]}
{Numerator} & {Denominator} \\
\hline
-3.63350182727758466E+12 & -3.99685201000472987E+13 \\
 5.09119767488377891E+11 & -2.82191995658341951E+13 \\
-4.66070769596744520E+10 & -9.73533805198387887E+12 \\
 1.82540967399093927E+09 & -2.18080501665434778E+12 \\
-5.30543200345946936E+07 & -3.55350240246925148E+11 \\
 & -4.46814329074875211E+10 \\
 & -4.48324401716056338E+09 \\
 & -3.65555177659264031E+08 \\
 & -2.43548935295006352E+07 \\
 & -1.31837759196109729E+06 \\
 & -5.61522957956144616E+04 \\
 & -1.72623902086203936E+03 \\
 & -3.61093692542213220E+01 \\
 &  1.00000000000000000E+00   
\end{tabular}
\end{ruledtabular}
\end{table}
\begin{table}[H]
\caption{Rational minimax coefficients for $F_{6}$ on $[0,x_0\rangle$}
\label{r_A_6_coeffs}
\begin{ruledtabular}
\begin{tabular}{S[table-format=+1.16e+2]S[table-format=+1.16e+2]}
{Numerator} & {Denominator} \\
\hline
 6.32362964316323446E+07 &  8.22071853610988647E+08 \\
-1.78304485335124420E+07 &  4.80666442210646271E+08 \\
 2.51567753132107927E+06 &  1.34959799873156815E+08 \\
-2.19143264882283100E+05 &  2.40771063816958245E+07 \\
 1.27344879285248770E+04 &  3.03870085803678862E+06 \\
-5.00986429135037998E+02 &  2.85166544683446997E+05 \\
 1.29601270953674587E+01 &  2.02723823300825666E+04 \\
-2.00754947326145896E-01 &  1.08337493002115488E+03 \\
 1.42097370844837178E-03 &  4.12117791883748035E+01 \\
 &  1.00000000000000000E+00   
\end{tabular}
\end{ruledtabular}
\end{table}
\begin{table}[H]
\caption{Rational minimax coefficients for $F_{7}$ on $[0,x_0\rangle$}
\label{r_A_7_coeffs}
\begin{ruledtabular}
\begin{tabular}{S[table-format=+1.16e+2]S[table-format=+1.16e+2]}
{Numerator} & {Denominator} \\
\hline
 1.75909119461509143E+11 &  2.63863679192234972E+12 \\
-2.31397098271936913E+10 &  1.98111328666843863E+12 \\
 1.68875828571559918E+09 &  7.31805354458430146E+11 \\
-5.64452771442680053E+07 &  1.76968955184222011E+11 \\
 1.17407664422063718E+06 &  3.14408799088726826E+10 \\
 &  4.36407011112015921E+09 \\
 &  4.91056122401636080E+08 \\
 &  4.58523651661037259E+07 \\
 &  3.59865965982106407E+06 \\
 &  2.40434947023941125E+05 \\
 &  1.32823525115988495E+04 \\
 &  6.68417454613384130E+02 \\
 &  2.04252475843400746E+01 \\
 &  1.00000000000000000E+00   
\end{tabular}
\end{ruledtabular}
\end{table}
\begin{table}[H]
\caption{Rational minimax coefficients for $F_{8}$ on $[0,x_0\rangle$}
\label{r_A_8_coeffs}
\begin{ruledtabular}
\begin{tabular}{S[table-format=+1.16e+2]S[table-format=+1.16e+2]}
{Numerator} & {Denominator} \\
\hline
 4.27633852210096255E+07 &  7.26977548756867415E+08 \\
-1.31313288604750197E+07 &  4.27221005649961331E+08 \\
 1.92047093279423654E+06 &  1.20645561717991569E+08 \\
-1.71296882686810612E+05 &  2.16664004023517011E+07 \\
 1.01020189016332066E+04 &  2.75570803292488760E+06 \\
-4.01157829572622705E+02 &  2.61012457722551099E+05 \\
 1.04348508576747219E+01 &  1.87669023892614156E+04 \\
-1.62054987695090529E-01 &  1.01792753889118730E+03 \\
 1.14734845393421036E-03 &  3.94513904309799494E+01 \\
 &  1.00000000000000000E+00   
\end{tabular}
\end{ruledtabular}
\end{table}
\begin{table}[H]
\caption{Rational minimax coefficients for $F_{9}$ on $[0,x_0\rangle$}
\label{r_A_9_coeffs}
\begin{ruledtabular}
\begin{tabular}{S[table-format=+1.16e+2]S[table-format=+1.16e+2]}
{Numerator} & {Denominator} \\
\hline
 3.66090812345166521E+07 &  6.95572543455501572E+08 \\
-1.15680315739775994E+07 &  4.09534939434717099E+08 \\
 1.71937185342050874E+06 &  1.15897974071192602E+08 \\
-1.55021217572956428E+05 &  2.08647896181599307E+07 \\
 9.21144155752769637E+03 &  2.66133029522484206E+06 \\
-3.67768247155539817E+02 &  2.52929651302264212E+05 \\
 9.60302999038807409E+00 &  1.82610915753483673E+04 \\
-1.49531416596628179E-01 &  9.95830569569992250E+02 \\
 1.06048517076778994E-03 &  3.88529933329081858E+01 \\
 &  1.00000000000000000E+00   
\end{tabular}
\end{ruledtabular}
\end{table}
\begin{table}[H]
\caption{Rational minimax coefficients for $F_{10}$ on $[0,x_0\rangle$}
\label{r_A_10_coeffs}
\begin{ruledtabular}
\begin{tabular}{S[table-format=+1.16e+2]S[table-format=+1.16e+2]}
{Numerator} & {Denominator} \\
\hline
 3.18727019529535358E+07 &  6.69326741011720691E+08 \\
-1.03030212432610360E+07 &  3.94760969620761552E+08 \\
 1.55322739375449943E+06 &  1.11934472548745699E+08 \\
-1.41423769641528486E+05 &  2.01961126182670577E+07 \\
 8.46337556751722184E+03 &  2.58269289943073100E+06 \\
-3.39673496567585546E+02 &  2.46205845498650818E+05 \\
 8.90365945622176397E+00 &  1.78414280943905329E+04 \\
-1.39028066499372714E-01 &  9.77560143690833819E+02 \\
 9.87898143632479994E-04 &  3.83650317003543676E+01 \\
 &  1.00000000000000000E+00   
\end{tabular}
\end{ruledtabular}
\end{table}
\begin{table}[H]
\caption{Rational minimax coefficients for $F_{11}$ on $[0,x_0\rangle$}
\label{r_A_11_coeffs}
\begin{ruledtabular}
\begin{tabular}{S[table-format=+1.16e+2]S[table-format=+1.16e+2]}
{Numerator} & {Denominator} \\
\hline
 2.81728640639121210E+07 &  6.47975873469703400E+08 \\
-9.27753783966489735E+06 &  3.82754433300979251E+08 \\
 1.41622089785473823E+06 &  1.08717435139479832E+08 \\
-1.30119646526112459E+05 &  1.96542366007654340E+07 \\
 7.83983535586110006E+03 &  2.51910102226997192E+06 \\
-3.16281224293270705E+02 &  2.40783937925615525E+05 \\
 8.32354807569940454E+00 &  1.75044646089453221E+04 \\
-1.30365690643806388E-01 &  9.62973033666935387E+02 \\
 9.28460701879997724E-04 &  3.79830794453759681E+01 \\
 &  1.00000000000000000E+00   
\end{tabular}
\end{ruledtabular}
\end{table}
\begin{table}[H]
\caption{Rational minimax coefficients for $F_{12}$ on $[0,x_0\rangle$}
\label{r_A_12_coeffs}
\begin{ruledtabular}
\begin{tabular}{S[table-format=+1.16e+2]S[table-format=+1.16e+2]}
{Numerator} & {Denominator} \\
\hline
 2.53544601714217041E+07 &  6.33861504285287492E+08 \\
-8.48258198675025220E+06 &  3.74844250615203673E+08 \\
 1.30979027926005474E+06 &  1.06606600921230004E+08 \\
-1.21399979536359709E+05 &  1.93004584684968576E+07 \\
 7.36540787890653295E+03 &  2.47784305031495735E+06 \\
-2.98819753533120450E+02 &  2.37294882532379883E+05 \\
 7.90065371972972126E+00 &  1.72900919399835697E+04 \\
-1.24223182966353277E-01 &  9.53839898338405204E+02 \\
 8.87602463343472957E-04 &  3.77540526654878690E+01 \\
 &  1.00000000000000000E+00   
\end{tabular}
\end{ruledtabular}
\end{table}
\begin{table}[H]
\caption{Rational minimax coefficients for $F_{13}$ on $[0,x_0\rangle$}
\label{r_A_13_coeffs}
\begin{ruledtabular}
\begin{tabular}{S[table-format=+1.16e+2]S[table-format=+1.16e+2]}
{Numerator} & {Denominator} \\
\hline
 2.30819576500698694E+07 &  6.23212856551669737E+08 \\
-7.82628187112665194E+06 &  3.68923049044818432E+08 \\
 1.22075808419529548E+06 &  1.05041400991127738E+08 \\
-1.14056644748954035E+05 &  1.90411719871711305E+07 \\
 6.96500421338905853E+03 &  2.44805073752281911E+06 \\
-2.84102269345918437E+02 &  2.34824627705903913E+05 \\
 7.54567809731318183E+00 &  1.71426120081067911E+04 \\
-1.19100077239224798E-01 &  9.47807814143492262E+02 \\
 8.53805650125360363E-04 &  3.76219983903593232E+01 \\
 &  1.00000000000000000E+00   
\end{tabular}
\end{ruledtabular}
\end{table}
\begin{table}[H]
\caption{Rational minimax coefficients for $F_{14}$ on $[0,x_0\rangle$}
\label{r_A_14_coeffs}
\begin{ruledtabular}
\begin{tabular}{S[table-format=+1.16e+2]S[table-format=+1.16e+2]}
{Numerator} & {Denominator} \\
\hline
 2.12458334186947741E+07 &  6.16129169141965232E+08 \\
-7.28795130204511162E+06 &  3.65028312420485399E+08 \\
 1.14728519794468474E+06 &  1.04025946459673529E+08 \\
-1.08001312706828555E+05 &  1.88758341061380906E+07 \\
 6.63702808514035845E+03 &  2.42947540498253127E+06 \\
-2.72194232447144952E+02 &  2.33331119143442641E+05 \\
 7.26358339711147027E+00 &  1.70575049731055860E+04 \\
-1.15125375121206069E-01 &  9.44573172805824440E+02 \\
 8.28372975875056273E-04 &  3.75688900688652846E+01 \\
 &  1.00000000000000000E+00   
\end{tabular}
\end{ruledtabular}
\end{table}
\begin{table}[H]
\caption{Rational minimax coefficients for $F_{15}$ on $[0,x_0\rangle$}
\label{r_A_15_coeffs}
\begin{ruledtabular}
\begin{tabular}{S[table-format=+1.16e+2]S[table-format=+1.16e+2]}
{Numerator} & {Denominator} \\
\hline
 1.97951920905537712E+07 &  6.13650954807010558E+08 \\
-6.86144392935271214E+06 &  3.63755226051815399E+08 \\
 1.08945206014270508E+06 &  1.03722759914749749E+08 \\
-1.03307183554948528E+05 &  1.88324083020775319E+07 \\
 6.38895089014890834E+03 &  2.42547911704470786E+06 \\
-2.63500160865368274E+02 &  2.33109185881384500E+05 \\
 7.06742762934637086E+00 &  1.70537146317930562E+04 \\
-1.12538812296089342E-01 &  9.44992042708637598E+02 \\
 8.13253423623244395E-04 &  3.76184079908039309E+01 \\
 &  1.00000000000000000E+00   
\end{tabular}
\end{ruledtabular}
\end{table}
\begin{table}[H]
\caption{Rational minimax coefficients for $F_{16}$ on $[0,x_0\rangle$}
\label{r_A_16_coeffs}
\begin{ruledtabular}
\begin{tabular}{S[table-format=+1.16e+2]S[table-format=+1.16e+2]}
{Numerator} & {Denominator} \\
\hline
 1.86192784303380985E+07 &  6.14436188201034918E+08 \\
-6.51388986576138143E+06 &  3.64367183315307283E+08 \\
 1.04240529459482231E+06 &  1.03940248825357695E+08 \\
-9.95208121677860958E+04 &  1.88798090295968553E+07 \\
 6.19202142390343984E+03 &  2.43258885879432064E+06 \\
-2.56771026974047403E+02 &  2.33882192253606616E+05 \\
 6.92127832841267607E+00 &  1.71157875559842136E+04 \\
-1.10719326466696887E-01 &  9.48495640903491750E+02 \\
 8.03538741614490102E-04 &  3.77642691208863824E+01 \\
 &  1.00000000000000000E+00   
\end{tabular}
\end{ruledtabular}
\end{table}
\begin{table}[H]
\caption{Rational minimax coefficients for $F_{17}$ on $[0,x_0\rangle$}
\label{r_A_17_coeffs}
\begin{ruledtabular}
\begin{tabular}{S[table-format=+1.16e+2]S[table-format=+1.16e+2]}
{Numerator} & {Denominator} \\
\hline
 1.76287634789397660E+07 &  6.17006721762789511E+08 \\
-6.21942096176012145E+06 &  3.65975273419925384E+08 \\
 1.00259907210026151E+06 &  1.04421803079143984E+08 \\
-9.63482365065391979E+04 &  1.89710288645764397E+07 \\
 6.03043257975913718E+03 &  2.44473111450708740E+06 \\
-2.51454186523567315E+02 &  2.35070543543347768E+05 \\
 6.81322696418388913E+00 &  1.72022767143530895E+04 \\
-1.09530227762568917E-01 &  9.52970348100220882E+02 \\
 7.98685174729464649E-04 &  3.79263609084777770E+01 \\
 &  1.00000000000000000E+00   
\end{tabular}
\end{ruledtabular}
\end{table}
\begin{table}[H]
\caption{Rational minimax coefficients for $F_{18}$ on $[0,x_0\rangle$}
\label{r_A_18_coeffs}
\begin{ruledtabular}
\begin{tabular}{S[table-format=+1.16e+2]S[table-format=+1.16e+2]}
{Numerator} & {Denominator} \\
\hline
 1.10053495649937993E+06 &  4.07197933905399874E+07 \\
-4.15290284510376398E+05 &  2.32658583260025049E+07 \\
 7.16040579546255159E+04 &  6.34852227957287156E+06 \\
-7.36091152895557778E+03 &  1.09225406973565756E+06 \\
 4.92944440283867767E+02 &  1.31450200622355534E+05 \\
-2.19946789490235600E+01 &  1.15572307375223640E+04 \\
 6.37661526509535572E-01 &  7.47423567664348937E+02 \\
-1.09650093594191108E-02 &  3.40341812575986739E+01 \\
 8.54702189691105590E-05 &  1.00000000000000000E+00   
\end{tabular}
\end{ruledtabular}
\end{table}
\begin{table}[H]
\caption{Rational minimax coefficients for $F_{19}$ on $[0,x_0\rangle$}
\label{r_A_19_coeffs}
\begin{ruledtabular}
\begin{tabular}{S[table-format=+1.16e+2]S[table-format=+1.16e+2]}
{Numerator} & {Denominator} \\
\hline
 1.05789196387343878E+06 &  4.12577865911173575E+07 \\
-4.01973620618696332E+05 &  2.35682404272109034E+07 \\
 6.97546013607269779E+04 &  6.42907318962590525E+06 \\
-7.21444241392353212E+03 &  1.10563143724400098E+06 \\
 4.85960478493897660E+02 &  1.32976727424942939E+05 \\
-2.18064735206264471E+01 &  1.16807305638863958E+04 \\
 6.35745444451493209E-01 &  7.54304476480918715E+02 \\
-1.09927619511029222E-02 &  3.42787973526879301E+01 \\
 8.61606672959964407E-05 &  1.00000000000000000E+00   
\end{tabular}
\end{ruledtabular}
\end{table}
\begin{table}[H]
\caption{Rational minimax coefficients for $F_{20}$ on $[0,x_0\rangle$}
\label{r_A_20_coeffs}
\begin{ruledtabular}
\begin{tabular}{S[table-format=+1.16e+2]S[table-format=+1.16e+2]}
{Numerator} & {Denominator} \\
\hline
 6.47075356638978264E+06 &  2.65300896221719669E+08 \\
-2.11794040363122806E+06 &  1.66125763126255078E+08 \\
 3.09732774375819352E+05 &  5.02387297315930431E+07 \\
-2.62120805979412831E+04 &  9.71995268750334966E+06 \\
 1.38502012427518084E+03 &  1.34171144238381706E+06 \\
-4.56416106277303520E+01 &  1.39219189884935085E+05 \\
 8.67411673772646667E-01 &  1.11001533818694266E+04 \\
-7.32312148535088359E-03 &  6.80698925494075328E+02 \\
 &  3.03328267496977109E+01 \\
 &  1.00000000000000000E+00   
\end{tabular}
\end{ruledtabular}
\end{table}
\begin{table}[H]
\caption{Rational minimax coefficients for $F_{21}$ on $[0,x_0\rangle$}
\label{r_A_21_coeffs}
\begin{ruledtabular}
\begin{tabular}{S[table-format=+1.16e+2]S[table-format=+1.16e+2]}
{Numerator} & {Denominator} \\
\hline
 1.00113806692073362E+06 &  4.30489368776244012E+07 \\
-3.85384394542953371E+05 &  2.45641218262036544E+07 \\
 6.77086190185746604E+04 &  6.69125494568673421E+06 \\
-7.08719057795389421E+03 &  1.14861041840245552E+06 \\
 4.83026628802106735E+02 &  1.37810035550222365E+05 \\
-2.19284174480561886E+01 &  1.20651381624713644E+04 \\
 6.46767482399436452E-01 &  7.75299072818440040E+02 \\
-1.13143654935771093E-02 &  3.49985539072948972E+01 \\
 8.97258365399553749E-05 &  1.00000000000000000E+00   
\end{tabular}
\end{ruledtabular}
\end{table}
\begin{table}[H]
\caption{Rational minimax coefficients for $F_{22}$ on $[0,x_0\rangle$}
\label{r_A_22_coeffs}
\begin{ruledtabular}
\begin{tabular}{S[table-format=+1.16e+2]S[table-format=+1.16e+2]}
{Numerator} & {Denominator} \\
\hline
 9.78413399959775489E+05 &  4.40286029982168180E+07 \\
-3.78882973996717940E+05 &  2.51053115917952085E+07 \\
 6.69529564811272573E+04 &  6.83266786797908781E+06 \\
-7.04831896020322737E+03 &  1.17158690304016054E+06 \\
 4.83136707558805601E+02 &  1.40366065850910693E+05 \\
-2.20607408482679038E+01 &  1.22656099930746168E+04 \\
 6.54514378702913320E-01 &  7.86052153038727682E+02 \\
-1.15190214816880878E-02 &  3.53527986648245085E+01 \\
 9.19134544479325137E-05 &  1.00000000000000000E+00   
\end{tabular}
\end{ruledtabular}
\end{table}
\begin{table}[H]
\caption{Rational minimax coefficients for $F_{23}$ on $[0,x_0\rangle$}
\label{r_A_23_coeffs}
\begin{ruledtabular}
\begin{tabular}{S[table-format=+1.16e+2]S[table-format=+1.16e+2]}
{Numerator} & {Denominator} \\
\hline
 9.66613236589941587E+05 &  4.54308221197475695E+07 \\
-3.76591537940241924E+05 &  2.58767005646896380E+07 \\
 6.69490245054781034E+04 &  7.03330268657338671E+06 \\
-7.09048327100888247E+03 &  1.20401692994392062E+06 \\
 4.89000278904994944E+02 &  1.43952347762249175E+05 \\
-2.24678080217337064E+01 &  1.25448930369473603E+04 \\
 6.70852908877587124E-01 &  8.00903396209939783E+02 \\
-1.18840290427715872E-02 &  3.58343785730047088E+01 \\
 9.54640048843181125E-05 &  1.00000000000000000E+00   
\end{tabular}
\end{ruledtabular}
\end{table}
\begin{table}[H]
\caption{Rational minimax coefficients for $F_{24}$ on $[0,x_0\rangle$}
\label{r_A_24_coeffs}
\begin{ruledtabular}
\begin{tabular}{S[table-format=+1.16e+2]S[table-format=+1.16e+2]}
{Numerator} & {Denominator} \\
\hline
 9.46417970705439261E+05 &  4.63744805645838129E+07 \\
-3.70615455935501387E+05 &  2.63957161413426991E+07 \\
 6.62217473689567327E+04 &  7.16819577499412023E+06 \\
-7.04931569194996637E+03 &  1.22578789705691770E+06 \\
 4.88689126513939576E+02 &  1.46353578064983032E+05 \\
-2.25733788136891364E+01 &  1.27310545374895777E+04 \\
 6.77724133763032956E-01 &  8.10731450762346773E+02 \\
-1.20745042392986278E-02 &  3.61468162397566562E+01 \\
 9.75714588013979498E-05 &  1.00000000000000000E+00   
\end{tabular}
\end{ruledtabular}
\end{table}
\begin{table}[H]
\caption{Rational minimax coefficients for $F_{25}$ on $[0,x_0\rangle$}
\label{r_A_25_coeffs}
\begin{ruledtabular}
\begin{tabular}{S[table-format=+1.16e+2]S[table-format=+1.16e+2]}
{Numerator} & {Denominator} \\
\hline
 9.38339083663866324E+05 &  4.78552932668707994E+07 \\
-3.69485294562541889E+05 &  2.72056831197545568E+07 \\
 6.63890042488980107E+04 &  7.37743864947674801E+06 \\
-7.10752303361419887E+03 &  1.25933265029253328E+06 \\
 4.95630940996386078E+02 &  1.50025416335256253E+05 \\
-2.30343254006395079E+01 &  1.30132038859781581E+04 \\
 6.95983183680208910E-01 &  8.25467951324447237E+02 \\
-1.24825859151127786E-02 &  3.66070109940675643E+01 \\
 1.01571653489208039E-04 &  1.00000000000000000E+00   
\end{tabular}
\end{ruledtabular}
\end{table}
\begin{table}[H]
\caption{Rational minimax coefficients for $F_{26}$ on $[0,x_0\rangle$}
\label{r_A_26_coeffs}
\begin{ruledtabular}
\begin{tabular}{S[table-format=+1.16e+2]S[table-format=+1.16e+2]}
{Numerator} & {Denominator} \\
\hline
 4.10714903877451297E+05 &  2.17678899055620140E+07 \\
-1.48284967299858010E+05 &  1.31172270011494235E+07 \\
 2.39347004924575201E+04 &  3.78861677215369447E+06 \\
-2.23771329927588588E+03 &  6.92924225284179908E+05 \\
 1.30765178330923992E+02 &  8.91172920829389735E+04 \\
-4.77083653277736153E+00 &  8.42912591432377850E+03 \\
 1.00472662976727574E-01 &  5.91980063956890030E+02 \\
-9.40521491916329223E-04 &  2.95952627591546617E+01 \\
 &  1.00000000000000000E+00   
\end{tabular}
\end{ruledtabular}
\end{table}
\begin{table}[H]
\caption{Rational minimax coefficients for $F_{27}$ on $[0,x_0\rangle$}
\label{r_A_27_coeffs}
\begin{ruledtabular}
\begin{tabular}{S[table-format=+1.16e+2]S[table-format=+1.16e+2]}
{Numerator} & {Denominator} \\
\hline
 4.08900785605349203E+05 &  2.24895432083406879E+07 \\
-1.48431482634856935E+05 &  1.35367048804001308E+07 \\
 2.40923408312081233E+04 &  3.90439736762716113E+06 \\
-2.26555506669706903E+03 &  7.12890261484845715E+05 \\
 1.33199027904202880E+02 &  9.14892435859503544E+04 \\
-4.89079482854629149E+00 &  8.62953599658876938E+03 \\
 1.03694945043714753E-01 &  6.03691417113411215E+02 \\
-9.77589742403062657E-04 &  3.00346183030949037E+01 \\
 &  1.00000000000000000E+00   
\end{tabular}
\end{ruledtabular}
\end{table}
\begin{table}[H]
\caption{Rational minimax coefficients for $F_{28}$ on $[0,x_0\rangle$}
\label{r_A_28_coeffs}
\begin{ruledtabular}
\begin{tabular}{S[table-format=+1.16e+2]S[table-format=+1.16e+2]}
{Numerator} & {Denominator} \\
\hline
 4.04918508607377351E+05 &  2.30803549906603548E+07 \\
-1.47688282537921989E+05 &  1.38797379677888631E+07 \\
 2.40900411595833825E+04 &  3.99892622045677261E+06 \\
-2.27701541273329491E+03 &  7.29154678562026655E+05 \\
 1.34599432591312964E+02 &  9.34152675146817121E+04 \\
-4.97062429847049156E+00 &  8.79149602027818586E+03 \\
 1.06031018404659362E-01 &  6.13092658398733545E+02 \\
-1.00609923918133880E-03 &  3.03809171630904679E+01 \\
 &  1.00000000000000000E+00   
\end{tabular}
\end{ruledtabular}
\end{table}
\begin{table}[H]
\caption{Rational minimax coefficients for $F_{29}$ on $[0,x_0\rangle$}
\label{r_A_29_coeffs}
\begin{ruledtabular}
\begin{tabular}{S[table-format=+1.16e+2]S[table-format=+1.16e+2]}
{Numerator} & {Denominator} \\
\hline
 4.02736747320286231E+05 &  2.37614680919304362E+07 \\
-1.47596908552103088E+05 &  1.42741859569694491E+07 \\
 2.41953671168406541E+04 &  4.10730483171108589E+06 \\
-2.29898080725758501E+03 &  7.47738452979963319E+05 \\
 1.36653819038177139E+02 &  9.56069803134792050E+04 \\
-5.07636744387583698E+00 &  8.97485315497558605E+03 \\
 1.08970343078371462E-01 &  6.23667277941326814E+02 \\
-1.04094686501161739E-03 &  3.07651138887207658E+01 \\
 &  1.00000000000000000E+00   
\end{tabular}
\end{ruledtabular}
\end{table}
\begin{table}[H]
\caption{Rational minimax coefficients for $F_{30}$ on $[0,x_0\rangle$}
\label{r_A_30_coeffs}
\begin{ruledtabular}
\begin{tabular}{S[table-format=+1.16e+2]S[table-format=+1.16e+2]}
{Numerator} & {Denominator} \\
\hline
 3.97919751441660774E+05 &  2.42731048379708781E+07 \\
-1.46422871477932610E+05 &  1.45707349185855753E+07 \\
 2.41040387989907886E+04 &  4.18883065073214594E+06 \\
-2.30043866634744442E+03 &  7.61720137240172586E+05 \\
 1.37381681867259204E+02 &  9.72552553702851800E+04 \\
-5.12892484415127183E+00 &  9.11255555596752242E+03 \\
 1.10687832961615384E-01 &  6.31587400459517944E+02 \\
-1.06341785154454686E-03 &  3.10500825233904551E+01 \\
 &  1.00000000000000000E+00   
\end{tabular}
\end{ruledtabular}
\end{table}
\begin{table}[H]
\caption{Rational minimax coefficients for $F_{31}$ on $[0,x_0\rangle$}
\label{r_A_31_coeffs}
\begin{ruledtabular}
\begin{tabular}{S[table-format=+1.16e+2]S[table-format=+1.16e+2]}
{Numerator} & {Denominator} \\
\hline
 3.91798219424327906E+05 &  2.46832878237592700E+07 \\
-1.44678734533971131E+05 &  1.48090417666897102E+07 \\
 2.39038910814725430E+04 &  4.25449476171968713E+06 \\
-2.29005841777242621E+03 &  7.73005948653021406E+05 \\
 1.37314139515175157E+02 &  9.85883385117327141E+04 \\
-5.14842853033226852E+00 &  9.22410511306109881E+03 \\
 1.11618796618234802E-01 &  6.38010373487974454E+02 \\
-1.07763143274847864E-03 &  3.12807829393761209E+01 \\
 &  1.00000000000000000E+00   
\end{tabular}
\end{ruledtabular}
\end{table}
\begin{table}[H]
\caption{Rational minimax coefficients for $F_{32}$ on $[0,x_0\rangle$}
\label{r_A_32_coeffs}
\begin{ruledtabular}
\begin{tabular}{S[table-format=+1.16e+2]S[table-format=+1.16e+2]}
{Numerator} & {Denominator} \\
\hline
 3.83366575753585584E+05 &  2.49188274240078582E+07 \\
-1.41962912506031840E+05 &  1.49473925142764407E+07 \\
 2.35226110455187173E+04 &  4.29304258040583941E+06 \\
-2.26023770325155337E+03 &  7.79706453088844968E+05 \\
 1.35947943805136022E+02 &  9.93889542962424737E+04 \\
-5.11394118864555818E+00 &  9.29189196777588754E+03 \\
 1.11257488030836662E-01 &  6.41960930347546819E+02 \\
-1.07813149210223792E-03 &  3.14246148951811622E+01 \\
 &  1.00000000000000000E+00   
\end{tabular}
\end{ruledtabular}
\end{table}

\end{document}